\documentclass[11pt]{article}
\usepackage{natbib}
\usepackage{amsmath,amsthm,amsopn,amstext,amscd,amsfonts,amssymb,color,enumerate}

\def\tr{\mathop{\rm tr}\nolimits}
\def\erf{\mathop{\rm erf}\nolimits}
\def\Cov{\mathop{\rm Cov}\nolimits}
\def\Var{\mathop{\rm Var}\nolimits}
\def\Vec{\mathop{\rm vec}\nolimits}
\def\Vech{\mathop{\rm vech}\nolimits}
\def\v{\mathop{\rm v}\nolimits}
\def \E {\mathop{\rm E}\nolimits}
\def \P {\mathop{\rm P}\nolimits}
\def \build#1#2#3{\mathrel{\mathop{#1}\limits^{#2}_{#3}}}

\newtheorem{thm}{Theorem}[section]

\newtheorem{lem}{Lemma}[section]
\theoremstyle{definition}

\newtheorem{rem}{Remark}[section]

\newcommand{\ams}[2]
           {\begin{center}
            \begin{minipage}{5.25in}
            \small
            \noindent \textbf{Mathematics Subject Classification: }{\uppercase{#1}}
            \end{minipage}
            \end{center}
            \par\normalsize
           }

\newcommand{\keywords}[1]
           {\begin{center}
            \begin{minipage}{5.25in}
            \small
            \noindent \textbf{Key Words:}~{\textrm{#1}}
            \end{minipage}
            \end{center}
            \normalsize
           }

\title{\vspace{-2.5cm}\textbf{\Large A modified Prékopa's  approach in optimum allocation in multivariate stratified random sampling}}
\author{
  \begin{normalsize}
  \begin{tabular}{c}
    \textbf{ Jos\'e A. D\'\i az-Garc\'\i a} \\
    Department of Statistics and Computation \\
    Universidad Aut\'onoma Agraria Antonio Narro \\
    25350 Buenavista, Saltillo, Coahuila, M\'exico. \\
    jadiaz@uaaan.mx \\
    and\\
    \textbf{Rogelio Ramos-Quiroga}\\
    Centro de Investigaci\'on en Matem\'aticas\\
    Department of Probability and Statistics\\
    Callej\'on de Jalisco s/n.\\
    36240 Guanajuato, M\'exico\\
    rramosq@cimat.mx\\
  \end{tabular}
  \end{normalsize}
}
\date{}

\begin{document}
  \maketitle
\begin{abstract}
\noindent A modified Pr\'ekopa's approach is considered for the problem of optimum allocation
in multivariate stratified random sampling. An example is solved by applying the proposed
methodology.
\end{abstract}


\ams{62D05, 90C15, 90C29, 90C10}

\section{Introduction}\label{sec1}

One of the statistical tools most commonly used in many fields of scientific research is the
theory of probabilistic sampling. In diverse practical situations, the probabilistic model of
stratified random sampling is frequently applied. Although there are different ways to
allocate the sample in strata, the optimum allocation has been found to be a useful approach,
see \citep{s54}, \cite{coc77}, \cite{sssa84} and \cite{t97}.

From a multivariate point of view, there are, basically, two approaches for solving the problem of optimum
allocation in multivariate stratified random sampling. When a cost function is defined as the objective function
subject to certain functions of variances to be within a given region, the problem of the optimum allocation
in multivariate stratified random sampling is stated as a deterministic uniobjective
mathematical programming problem, see \citet{ad81} among others. Alternatively, when the
objective function is defined as some functions of variances subject to cost restrictions, the
problem has been proposed implicitly and explicitly as a deterministic multiobjective
mathematical programming problem, see \cite{coc77}, \cite{sssa84} and \citet{dgu:08}.

On the other hand, \citet{pre78} considers the approach wherein population variances are random
variables and formulated the corresponding optimum allocation problem as a stochastic (or
probabilistic) mathematical programming problem, termed specifically chance constraints
approach, see \citet{chc:63}. Namely, \citet{pre78} minimizes a cost function subject to
inequality restrictions in terms of the estimated variance of the stratified mean of each
characteristic, restrictions that are allowed to be violated with certain probability. An
alternative approach suggested by \citet{dgu:08} is developed by \citet{kw:10} from a
stochastic point of view.

This work states the optimum allocation in multivariate stratified random sampling as a
stochastic integer programming problem, specifically, a modified Pr\'ekopa's
approach is proposed. Section \ref{sec2} includes some notation and definitions on multivariate
stratified random sampling and summarizes properties on the asymptotic normality of the sample
covariance matrices. The optimum allocation in multivariate stratified random sampling via
chance constraints methodology is studied in Section \ref{sec3}. Finally an application of the
approach is presented in Section \ref{sec4}.

\section{Preliminaries on multivariate stratified random sampling}\label{sec2}

Consider a population of size $N$, divided into $H$ sub-populations (strata). We wish to find a
representative sample of size $n$ and an optimum allocation rule for the strata, meeting the following
requirements: i) to minimize the variance of the estimated mean, subject to a budgetary
constraint; or ii) to minimize the cost subject to a constraint on the variances; this is the
classical problem in optimum allocation in univariate stratified sampling, see \cite{coc77},
\cite{sssa84} and \cite{t97}. However, if more than one characteristic (variable) is being considered,
then the problem is known as optimum allocation in multivariate stratified sampling. For a
formal expression of the problem of optimum allocation in multivariate stratified sampling, consider the
following notation.

\subsection{Notation}

The subindex $h=1,2,\ldots,H$ denotes the stratum, $i=1,2,\ldots,N_{h} \mbox{ or } n_{h}$ the unit
within stratum $h$ and $j=1,2,\ldots,G$ denotes the characteristic (variable). Moreover:

\bigskip

\begin{footnotesize}
\begin{tabular}{ll}
    $N_{h}$ & Total number of units within stratum $h$\\
    $n_{h}$ &  Number of units from the sample in stratum $h$\\
    \begin{tabular}{lcl}
       $\mathbf{Y}_{h}$ &=& $(\mathbf{Y}_{h}^{1} \dots \mathbf{Y}_{h}^{G})$ \\
        &=& $(\mathbf{Y}_{h1} \dots \mathbf{Y}_{h N_{h}})'$
     \end{tabular} & $N_{h} \times G$ matrix population in stratum $h$; $\mathbf{Y}_{hi}$ is the\\
     & $i$-th $G$-dimensional value of the $i$-th unit in stratum $h$\\
    \begin{tabular}{lcl}
       $\mathbf{y}_{h}$ &=& $(\mathbf{y}_{h}^{1} \dots \mathbf{y}_{h}^{G})$ \\
        &=& $(\mathbf{y}_{h1} \dots \mathbf{y}_{h n_{h}})'$
     \end{tabular} & $n_{h} \times G$ matrix sample in stratum $h$; $\mathbf{y}_{hi}$ is the\\
     & $i$-th element of the $G$-dimensional random sample\\
     & in stratum $h$\\
    $y_{hi}^{j}$ &  Value obtained for the $i$-th unit in stratum $h$\\
    & of the $j$-th characteristic\\[1ex]
    $\mathbf{n} = ({n}_{1},\dots, {n}_{H})'$ & Vector of the number of units in the sample\\
    $\displaystyle{W_{h}} = \displaystyle{\frac{N_{h}}{N}}$ & Relative size of stratum  $h$\\[3ex]
    $\displaystyle{\overline{Y}_{h}^{j}} = \frac{1}{N_{h}} \displaystyle{
        \sum_{i=1}^{N_{h}}} y_{hi}^j$ & Population mean in stratum $h$ of the $j$-th characteristic\\[3ex]
    $\overline{\mathbf{Y}}_{h}=  (\overline{Y}_{h}^{1}, \dots,\overline{Y}_{h}^{G})'$
    & Population mean vector in stratum $h$ \\[1ex]
\end{tabular}
\end{footnotesize}

\begin{footnotesize}
\begin{tabular}{ll}
    $\displaystyle{\overline{y}_{h}^{j}}= \frac{1}{n_h} \displaystyle{
        \sum_{i=1}^{n_{h}}} y_{hi}^{j}$ & Sample mean in stratum $h$ of the $j$-th characteristic\\[3ex]
    $\overline{\mathbf{y}}_{h}=  (\overline{y}_{h}^{1}, \dots,\overline{y}_{h}^{G})'$
    & Sample mean vector in stratum $h$ \\[1ex]
    $\displaystyle{\overline{y}_{_{ST}}^{j}=  \sum_{h=1}^{H}W_{h}\overline{y}_{h}^{j}}
    $
    & Estimator of the population mean in multivariate\\
    & stratified sampling for the $j$-th characteristic\\
    $\overline{\mathbf{y}}_{_{ST}}=(\overline{y}_{_{ST}}^{1},\dots,\overline{y}_{_{ST}}^{G})'$
    & Estimator of the population mean vector in\\
    & multivariate stratified sampling\\
    $\mathbf{S}_{h}$ & Variance-covariance matrix in stratum $h$\\
    & $\mathbf{S}_{h}= \displaystyle{\frac{1}{N_h}} \displaystyle{
        \sum_{i=1}^{N_{h}}} (\mathbf{y}_{hi}-\overline{\mathbf{Y}}_{h})(\mathbf{y}_{hi}-\overline{\mathbf{Y}}_{h})'$\\[2ex]
    &  where $S_{h_{jk}}$ is the covariance in stratum $h$ of the\\
    & $j$-th and $k$-th characteristics; furthermore\\
    & $S_{h_{jk}}= \displaystyle{\frac{1}{N_h}}  \displaystyle{
        \sum_{i=1}^{N_{h}}}(y_{hi}^{j}-\overline{y}_{h}^{j})(y_{hi}^{k}-\overline{y}_{h}^{k})$, and \\
    & $S_{h_{jj}}\equiv S_{hj}^{2} = \displaystyle{\frac{1}{N_h}} \displaystyle{
        \sum_{i=1}^{N_{h}}}(y_{hi}^{j}-\overline{y}_{h}^{j})^2$\\
    $\mathbf{s}_{h}$ & Estimator of the covariance matrix in stratum\\
    &  $h$;\\
    & $\mathbf{s}_{h}=\displaystyle{\frac{1}{n_h-1}} \displaystyle{
        \sum_{i=1}^{n_{h}}}(\mathbf{y}_{hi}-\overline{\mathbf{y}}_{h})(\mathbf{y}_{hi}-\overline{\mathbf{y}}_{h})'$\\
    &  where $s_{h_{jk}}$ is the sample covariance in stratum $h$ of the\\
    & $j$-th and $k$-th characteristics; furthermore\\
    & $s_{h_{jk}}=\displaystyle{\frac{1}{n_h-1}} \displaystyle{
        \sum_{i=1}^{n_{h}}}(y_{hi}^{j}-\overline{y}_{h}^{j})(y_{hi}^{k}-\overline{y}_{h}^{k})$, and \\
    & $s_{h_{jj}}\equiv s_{hj}^{2} = \displaystyle{\frac{1}{n_h-1}} \displaystyle{
        \sum_{i=1}^{n_{h}}}(y_{hi}^{j}-\overline{y}_{h}^{j})^2$\\
    $\Cov(\overline{\mathbf{y}}_{_{ST}})$
        & Variance-covariance matrix of $\overline{\mathbf{y}}_{_{ST}}$\\
    $
    \widehat{\Cov}(\overline{\mathbf{y}}_{_{ST}})$
        & Estimator of the covariance matrix of $\overline{\mathbf{y}}_{_{ST}}$,\\
        & it is denoted as $\widehat{\Cov}(\overline{\mathbf{y}}_{_{ST}}) \equiv
        \widehat{\Cov(\overline{\mathbf{y}}_{_{ST}})}$, and defined as \\[1ex]
     & $
             = \left(
                \begin{array}{cccc}
                  \widehat{\Var}(\overline{y}_{_{ST}}^{1}) & \widehat{\Cov}(\overline{y}_{_{ST}}^{1},\overline{y}_{_{ST}}^{2})
                    & \cdots & \widehat{\Cov}(\overline{y}_{_{ST}}^{1},\overline{y}_{_{ST}}^{G}) \\
                  \widehat{\Cov}(y_{_{ST}}^{2},\overline{y}_{_{ST}}^{1}) & \widehat{\Var}(\overline{y}_{_{ST}}^{2}) & \cdots
                    & \widehat{\Cov}(\overline{y}_{_{ST}}^{2},\overline{y}_{_{ST}}^{G}) \\
                  \vdots & \vdots & \ddots & \vdots \\
                  \widehat{\Cov}(\overline{y}_{_{ST}}^{G},\overline{y}_{_{ST}}^{1}) & \widehat{\Cov}(\overline{y}_{_{ST}}^{G},
                    \overline{y}_{_{ST}}^{2}) & \cdots & \widehat{\Var}(\overline{y}_{_{ST}}^{G}) \\
                \end{array}
            \right )
          $\\[3ex]
        & = $\displaystyle{\sum_{h=1}^{H}\frac{{{W_{h}}^{2}}
        \mathbf{s}_{h}}{n_{h}} - \sum_{h=1}^{H} \frac{{W_{h}}\mathbf{s}_{h}}{N}}$\\[1ex]
        $\widehat{\Cov}(\overline{y}_{_{ST}}^{j},\overline{y}_{_{ST}}^{k}) $ &
        Estimated covariance of $\overline{y}_{_{ST}}^{j}$ and $\overline{y}_{_{ST}}^{k}$ where \\[1ex]
        & $\widehat{\Cov}(\overline{y}_{_{ST}}^{k},\overline{y}_{_{ST}}^{j}) \equiv
        \widehat{\Cov(\overline{y}_{_{ST}}^{j},\overline{y}_{_{ST}}^{k})}$, with \\[1ex]
        &  $\widehat{\Cov}(\overline{y}_{_{ST}}^{j},\overline{y}_{_{ST}}^{k})= \displaystyle{\sum_{h=1}^{H}\frac{{{W_{h}}^{2}}
        s_{h_{jk}}}{n_{h}} - \sum_{h=1}^{H} \frac{{W_{h}}s_{h_{jk}}}{N}}$, and \\[1ex]
        &
        $\widehat{\Cov}(\overline{y}_{_{ST}}^{j},\overline{y}_{_{ST}}^{j}) \equiv
        \widehat{\Var}(\overline{y}_{_{ST}}^{j})= \displaystyle{\sum_{h=1}^{H}\frac{{{W_{h}}^{2}}
        s_{hj}^{2}}{n_{h}} - \sum_{h=1}^{H} \frac{{W_{h}}s_{hj}^{2}}{N}}$ \\[2ex]
       $c_{h}$ & Cost per $G$-dimensional sampling unit in stratum $h$ and let\\
       & $\mathbf{c} = (c_{1}, \dots, c_{G})'$.\\[2ex]
\end{tabular}
\end{footnotesize}

\noindent Where if $\mathbf{a} \in \Re^{G}$, $\mathbf{a}'$ denotes the transpose of
$\mathbf{a}$.

\subsection{Asymptotic normality}

Now, the asymptotic distribution of the estimator, $\mathbf{s}_{h}$, of the covariance matrix  is
stated. First, consider the following notation and definitions.

A detailed discussion of operator ``$\Vec$", ``$\Vech$", Moore-Penrose inverse, Kronecker
product, commutation matrix and duplication matrix may be found in \citet{mn:88}, among many
others. For convenience, some notation is introduced, although in general it adheres to
standard notation.

For all matrix $\mathbf{A}$ there exists a unique matrix $\mathbf{A}^{+}$ which is termed
\emph{Moore-Penrose inverse} of $\mathbf{A}$.

Let $\mathbf{A}$ be an $m \times n$ matrix and $\mathbf{B}$ a $p \times q$ matrix. The $mp
\times nq$ matrix defined by
$$
  \left[
  \begin{array}{ccc}
    a_{11}\mathbf{B} & \cdots & a_{11}\mathbf{B} \\
    \vdots & \ddots & \vdots \\
    a_{11}\mathbf{B} & \cdots & a_{11}\mathbf{B}
  \end{array}
  \right]
$$
is termed the \emph{Kronecker product} (also termed tensor product or direct product) of
$\mathbf{A}$ and $\mathbf{B}$ and written $\mathbf{A} \otimes \mathbf{B}$. Let $\mathbf{C}$ be
an $m \times n$ matrix and $\mathbf{C}_{j}$ its $j$-th column, then $\Vec \mathbf{C}$ is the
$mn \times 1$ vector
$$
  \Vec C =
  \left [
  \begin{array}{c}
    \mathbf{C}_{1} \\
    \mathbf{C}_{2} \\
    \vdots \\
    \mathbf{C}_{n}
  \end{array}
  \right].
$$
The vectors $\Vec \mathbf{C}$ and $\Vec \mathbf{C}^{'}$ clearly contain the same $mn$
components, but in different order. Therefore there exists a unique $mn \times mn$ permutation
matrix which transforms $\Vec \mathbf{C}$ into $\Vec \mathbf{C}'$. This matrix is termed the
\emph{commutation matrix} and is denoted $\mathbf{K}_{mn}$ (If $m=n$, it is often written
$\mathbf{K}_{n}$ instead of $\mathbf{K}_{mn}$). Hence
$$
  \mathbf{K}_{mn} \Vec \mathbf{C} = \Vec \mathbf{C}'.
$$
Similarly, let $\mathbf{B}$ be a square $n \times n$ matrix. Then $\Vech \mathbf{B}$ (also
denoted as $\v(\mathbf{B})$) shall denote the $n(n+1)/2 \times 1$ vector that is obtained
from $\Vec \mathbf{B}$ by eliminating all supradiagonal elements of $\mathbf{B}$. If
$\mathbf{B} = \mathbf{B}'$, $\Vech \mathbf{B}$ contains only the distinct elements of
$\mathbf{B}$, then there exists a unique $n^{2} \times n(n+1)/2$ matrix termed \textit{duplication
matrix}, which is denoted by $\mathbf{D}_{n}$, such that $\mathbf{D}_{n}\Vech \mathbf{B} = \Vec
\mathbf{B}$ and $\mathbf{D}_{n}^{+}\Vec \mathbf{B} = \Vech \mathbf{B}$. Finally, note that
$(\Vech \mathbf{B})' \equiv \Vech' \mathbf{B}$.

Now, with the above mathematical tools and based in the extension given in \citet{h:61},
the multivariate version of H\'ajek's theorem is restated in terms of sampling theory terminology,
which is explained in detail in \citet{dgrr:11}.

\begin{lem}\label{lemma1}
Let $\mathbf{\mathbf{\Xi}}_{\nu}$ be a $G \times G$ symmetric random matrix defined as
$$
   \mathbf{\Xi}_{\nu} =  \frac{1}{n_{\nu}-1}\sum_{i = 1}^{n_{\nu}}(\mathbf{y}_{\nu i}- \overline{\mathbf{Y}}_{\nu})
   (\mathbf{y}_{\nu i}- \overline{\mathbf{Y}}_{\nu})'.
$$
Suppose that for $\boldsymbol{\lambda} = (\lambda_{1}, \dots, \lambda_{k})'$, any vector of constants, $k = G(G+1)/2$,%
{\small
\begin{equation}\label{shc}
  \boldsymbol{\lambda}'\left(\mathbf{M}_{\nu}^{4} - \Vech \mathbf{S}_{\nu}
  \Vech' \mathbf{S}_{\nu}\right) \boldsymbol{\lambda} \geq \epsilon \build{\max}{}{1 \leq \alpha \leq k
  } \left[\lambda_{\alpha}^{2} \mathbf{e}_{k}^{\alpha '}\left(\mathbf{M}_{\nu}^{4} - \Vech \mathbf{S}_{\nu}
  \Vech' \mathbf{S}_{\nu}\right) \mathbf{e}_{k}^{\alpha }\right],
\end{equation}}
where $\mathbf{e}_{k}^{\alpha } = (0, \dots, 0, 1, 0, \dots, 0)'$ is the $\alpha$-th vector of
the canonical base of $\Re^{k}$, $\epsilon > 0$ and independent of $\nu
> 1$ and
$$
  \mathbf{M}_{\nu}^{4} = \frac{1}{N_{\nu}}\mathbf{D}_{G}^{+}\left[\sum_{i = 1}^{N_{\nu}}
  (\mathbf{y}_{\nu i} - \overline{\mathbf{Y}}_{\nu})(\mathbf{y}_{\nu i} - \overline{\mathbf{Y}}_{\nu})'
  \otimes (\mathbf{y}_{\nu i} - \overline{\mathbf{Y}}_{\nu})(\mathbf{y}_{\nu i} - \overline{\mathbf{Y}}_{\nu})'
  \right]\mathbf{D}_{G}^{+'},
$$
is the fourth central moment. Assume that $n_{\nu}\rightarrow \infty$, $N_{\nu} - n_{\nu}
\rightarrow \infty$, $N_{\nu}\rightarrow \infty$, and that, for all $j = 1,\dots,G$,%
\begin{equation}\label{hcas}
  \left[\build{\lim}{}{\nu \rightarrow \infty}\left(\frac{n_{\nu}}{N_{\nu}}\right) = 0\right] \Rightarrow
  \build{\lim}{}{\nu \rightarrow \infty} \frac{\build{\max}{}{1 \leq i_{1} < \cdots < i_{n_{\nu}}\leq N_{\nu}}
  \displaystyle\sum_{\beta = 1}^{n_{\nu}}\left[\left(y_{\nu i_{\beta}}^{j} - \overline{Y}_{\nu}^{j}\right)^{2} - S_{\nu j}^{2}\right]^{2}}
  {N_{\nu}\left[m_{\nu j}^{4}-\left(S_{\nu j}^{2}\right)^{2}\right]} = 0,
\end{equation}
where
$$
  m_{\nu j}^{4} =\frac{1}{N_{\nu}}\sum_{i = 1}^{N_{\nu}}\left(y_{\nu i}^{j} -
  \overline{y}_{\nu}^{j}\right)^{4}.
$$
Then, $\Vech\mathbf{\Xi}_{\nu}$ is asymptotically normal distributed as
$$
  \Vech\mathbf{\Xi}_{\nu} \build{\rightarrow}{d}{} \mathcal{N}_{k}(\E(\Vech\mathbf{\Xi}_{\nu}),
  \Cov(\Vech\mathbf{\Xi}_{\nu})),
$$
with
\begin{equation}\label{mXi}
    \E(\Vech\mathbf{\Xi}_{\nu}) = \frac{n_{\nu}}{n_{\nu}-1}\Vech\mathbf{S}_{\nu},
\end{equation}
and
\begin{equation}\label{cmXi}
    \Cov(\Vech\mathbf{\Xi}_{\nu}) = \frac{n_{\nu}}{(n_{\nu} - 1)^{2}}\left(\mathbf{M}_{\nu}^{4}
     - \Vech \mathbf{S}_{\nu}\Vech' \mathbf{S}_{\nu}\right).
\end{equation}
$n_{\nu}$ is the sample size for a simple random sample from the $\nu$-th population of size
$N_{\nu}$.
\end{lem}

Then:

\begin{thm}\label{teo1}
Under assumptions in Lemma \ref{lemma1}, the sequence of sample covariance matrices
$\mathbf{s}_{\nu}$ are such that $\Vech \mathbf{s}_{\nu}$ has an asymptotical normal with
asymptotic mean and covariance matrix given by (\ref{mXi}) and (\ref{cmXi}), respectively.
\end{thm}
\begin{proof}
This follows immediately from Lemma \ref{lemma1}, only observe that
\begin{eqnarray*}
  \mathbf{s}_{\nu} &=&  \frac{1}{n_{\nu}-1}\sum_{i = 1}^{n_{\nu}}(\mathbf{y}_{\nu i}- \overline{\mathbf{y}}_{\nu})
   (\mathbf{y}_{\nu i}- \overline{\mathbf{y}}_{\nu})' \\
   &=& \mathbf{\Xi} - \frac{n_{\nu}}{n_{\nu}-1} (\overline{\mathbf{y}}_{\nu}- \overline{\mathbf{Y}}_{\nu})
   (\overline{\mathbf{y}}_{\nu }- \overline{\mathbf{Y}}_{\nu})',
\end{eqnarray*}
where
$$
  \frac{n_{\nu}}{n_{\nu}-1} \rightarrow 1 \quad \mbox{ and } \quad (\overline{\mathbf{y}}_{\nu}- \overline{\mathbf{Y}}_{\nu})
   (\overline{\mathbf{y}}_{\nu }- \overline{\mathbf{Y}}_{\nu})'\rightarrow 0 \quad \mbox{in
   probability}. \qquad\mbox{\qed}
$$
\end{proof}

\begin{rem}
Observe that it is possible to found the asymptotic distribution of $\Vec \mathbf{S}_{\nu}$, but
this asymptotic normal distribution is singular, because $\Cov(\Vec\mathbf{S}_{\nu})$ is
singular. This is due to the fact $\Cov(\Vec\mathbf{S}_{\nu})$ is the $G^{2} \times G^{2}$
covariance matrix in the asymptotic distribution of $\Vec \mathbf{S}_{\nu}$ and,
because $\mathbf{S}_{\nu}$ is symmetric, then $\Vec \mathbf{S}_{\nu}$ has repeated elements. In
this case, $\Vec\mathbf{S}_{\nu}$ is asymptotically normally distributed as (see \citet{mh:82})
$$
  \Vec\mathbf{S}_{\nu} \build{\rightarrow}{d}{} \mathcal{N}_{G^{2}}(\E(\Vec\mathbf{\Xi}_{\nu}),
  \Cov(\Vec\mathbf{\Xi}_{\nu})),
$$
where
$$
  \E(\Vec\mathbf{\Xi}_{\nu}) = \frac{n_{\nu}}{n_{\nu}-1}\Vec\mathbf{S}_{\nu},
$$
$$
  \Cov(\Vec\mathbf{\Xi}_{\nu}) = \frac{n_{\nu}}{(n_{\nu} - 1)^{2}}\left(\mathfrak{M}_{\nu}^{4}
  - \Vec \mathbf{S}_{\nu}\Vec' \mathbf{S}_{\nu}\right),
$$
and
$$
  \mathfrak{M}_{\nu}^{4} = \frac{1}{N_{\nu}}\left[\sum_{i = 1}^{N_{\nu}}
  (\mathbf{y}_{\nu i} - \overline{\mathbf{Y}}_{\nu})(\mathbf{y}_{\nu i} - \overline{\mathbf{Y}}_{\nu})'
  \otimes (\mathbf{y}_{\nu i} - \overline{\mathbf{Y}}_{\nu})(\mathbf{y}_{\nu i} - \overline{\mathbf{Y}}_{\nu})'
  \right]. \qquad \mbox{\qed}
$$
\end{rem}

The following assertion is an immediate consequence of Theorem \ref{teo1}.

\begin{thm}
Let $\widehat{\Cov}(\overline{\mathbf{y}}_{_{ST}})$ be the estimator of the covariance matrix
of $\overline{\mathbf{y}}_{ST}$, then
$$
  \Vech \widehat{\Cov}(\overline{\mathbf{y}}_{_{ST}}) = \sum_{h=1}^{H}\left(\frac{{{W_{h}}^{2}}}{n_{h}} -
  \frac{{W_{h}}}{N} \right)\Vech \mathbf{s}_{h}
$$
is asymptotically normally distributed; furthermore
\begin{equation}\label{normal}
    \Vech \widehat{\Cov}(\overline{\mathbf{y}}_{_{ST}}) \build{\rightarrow}{d}{} \mathcal{N}_{k}
  \left(\E\left(\Vech \widehat{\Cov}(\overline{\mathbf{y}}_{_{ST}})\right),
  \Cov\left(\Vech \widehat{\Cov}(\overline{\mathbf{y}}_{_{ST}})\right)\right),
\end{equation}
where
\begin{equation}\label{ecyst}
    \E\left(\Vech \widehat{\Cov}(\overline{\mathbf{y}}_{_{ST}})\right) =  \sum_{h=1}^{H}\left(
     \frac{{{W_{h}}^{2}}}{n_{h}} - \frac{{W_{h}}}{N} \right) \frac{n_{h}}{n_{h}-1}\Vech\mathbf{S}_{h},
\end{equation}
\begin{eqnarray}
    \Cov\left(\Vech \widehat{\Cov}(\overline{\mathbf{y}}_{_{ST}})\right) \hspace{8cm}\nonumber \\
    \label{ccyst}
    \phantom{xx}=\sum_{h=1}^{H}\left(
    \frac{{{W_{h}}^{2}}}{n_{h}} - \frac{{W_{h}}}{N} \right)^{2} \frac{n_{h}}{(n_{h}-1)^{2}}
    \left(\mathbf{M}_{h}^{4} - \Vech \mathbf{S}_{h}\Vech' \mathbf{S}_{h}\right),
\end{eqnarray}
and
$$
  \mathbf{M}_{h}^{4} = \frac{1}{N_{h}}\mathbf{D}_{G}^{+}\left[\sum_{i = 1}^{N_{h}}
    (\mathbf{y}_{h i} - \overline{\mathbf{Y}}_{h})(\mathbf{y}_{h i} - \overline{\mathbf{Y}}_{h})'
    \otimes (\mathbf{y}_{h i} - \overline{\mathbf{Y}}_{h})(\mathbf{y}_{h i} - \overline{\mathbf{Y}}_{h})'
    \right]\mathbf{D}_{G}^{+'}.
$$
\end{thm}

Finally, note that the asymptotically normal distributions of  $\Vech \mathbf{S}_{h}$, $\Vec
\widehat{\Cov}(\overline{\mathbf{y}}_{_{ST}})$ and $ \Vech
\widehat{\Cov}(\overline{\mathbf{y}}_{_{ST}})$ are in terms of the parameters
$\overline{\mathbf{Y}}_{h}$, $\Vech \mathbf{S}_{h}$, $\mathfrak{M}_{h}^{4}$ and
$\mathbf{M}_{h}^{4}$; then, from \citet[iv), pp. 388-389]{r:73}, approximations of asymptotic
distributions can be obtained, making the following substitutions
\begin{equation}\label{sus}
    \overline{\mathbf{Y}}_{h} \rightarrow \overline{\mathbf{y}}_{h}, \qquad \Vech \mathbf{S}_{h}
  \rightarrow\Vech \mathbf{s}_{h},  \quad \mathfrak{M}_{h}^{4} \rightarrow \mathfrak{m}_{h}^{4}
  \quad \mbox{ and } \quad \mathbf{M}_{h}^{4} \rightarrow \mathbf{m}_{h}^{4}
\end{equation}
where
$$
  \mathbf{m}_{h}^{4} = \frac{1}{n_{h}}\mathbf{D}_{G}^{+}\left[\sum_{i = 1}^{n_{h}}
    (\mathbf{y}_{h i} - \overline{\mathbf{y}}_{h})(\mathbf{y}_{h i} - \overline{\mathbf{y}}_{h})'
    \otimes (\mathbf{y}_{h i} - \overline{\mathbf{y}}_{h})(\mathbf{y}_{h i} - \overline{\mathbf{y}}_{h})'
    \right]\mathbf{D}_{G}^{+'},
$$
and
$$
  \mathfrak{m}_{h}^{4} = \frac{1}{n_{h}}\left[\sum_{i = 1}^{n_{h}}
    (\mathbf{y}_{h i} - \overline{\mathbf{y}}_{h})(\mathbf{y}_{h i} - \overline{\mathbf{y}}_{h})'
    \otimes (\mathbf{y}_{h i} - \overline{\mathbf{y}}_{h})(\mathbf{y}_{h i} - \overline{\mathbf{y}}_{h})'
    \right].
$$

\section{Modified Pr\'ekopa's approach}\label{sec3}

Optimum allocation in multivariate stratified random sampling was proposed as the following
deterministic mathematical programming problem
\begin{equation}\label{om2}
  \begin{array}{c}
    \build{\min}{}{\mathbf{n}}\mathbf{c}'\mathbf{n} + c_{0} \\
    \mbox{subject to}\\
    \widehat{\Var}(\overline{y}_{_{ST}}^{j}) \leq v_{0}^{j}, \ \ j=1,2,\dots, G\\
    2\leq n_{h}\leq N_{h}, \ \ h=1,2,\dots, H\\
    n_{h}\in \mathbb{N},
  \end{array}
\end{equation}
where $v_{0}^{j}$ are desired precisions assigned to the variances of the sample mean $
\widehat{\Var}(\overline{y}_{_{ST}}^{j})$, $j=1,2,\dots, G$. This approach has been treated in
detail by \cite{ad81}.

From a stochastic point of view of (\ref{om2}), \citet{pre78} proposes the following chance
constraints mathematical program
\begin{equation}\label{om3}
  \begin{array}{c}
    \build{\min}{}{\mathbf{n}}\mathbf{c}'\mathbf{n} + c_{0} \\
    \mbox{subject to}\\
    \P\left(\widehat{\Var}(\overline{y}_{_{ST}}^{j}) \leq v_{0}^{j}\right) \geq p_{0}, \ \ j=1,2,\dots, G\\
    2\leq n_{h}\leq N_{h}, \ \ h=1,2,\dots, H\\
    n_{h}\in \mathbb{N},
  \end{array}
\end{equation}
where $0 \leq p_{0} \leq 1$ is a specified probability.

The present work considers the following alternative chance constraints mathematical
programming problem
\begin{equation}\label{om4}
    \begin{array}{c}
    \build{\min}{}{\mathbf{n}} \mathbf{c}'\mathbf{n}+ c_{0}\\
    \mbox{subject to}\\
    \P\left(\widehat{\Cov}(\overline{\mathbf{y}}_{_{ST}}) < \mathbf{\Delta}\right) \geq p_{0}\\
    2\leq n_{h}\leq N_{h}, \ \ h=1,2,\dots, H\\
    n_{h}\in \mathbb{N}\\
    \Vech \widehat{\Cov}(\overline{\mathbf{y}}_{_{ST}}) \build{\rightarrow}{d}{} \mathcal{N}_{k}
  \left(\E\left(\Vech \widehat{\Cov}(\overline{\mathbf{y}}_{_{ST}})\right),
  \Cov\left(\Vech \widehat{\Cov}(\overline{\mathbf{y}}_{_{ST}})\right)\right),
  \end{array}
\end{equation}
where $\mathbf{\Delta} > \mathbf{0}$ is a constant matrix.

From \citet{dgu:08}, note that $\widehat{\Cov}(\overline{\mathbf{y}}_{_{ST}})$ is an explicit
function of $\mathbf{n}$, and so it must be denoted as  as
$\widehat{\Cov}(\overline{\mathbf{y}}_{_{ST}}) \equiv
\widehat{\Cov}(\overline{\mathbf{y}}_{_{ST}}(\mathbf{n}))$. In addition, assume that
$\widehat{\Cov}(\overline{\mathbf{y}}_{_{ST}}(\mathbf{n}))$ is a positive definite matrix for
all $\mathbf{n}$, $\widehat{\Cov}(\overline{\mathbf{y}}_{_{ST}}(\mathbf{n})) > \mathbf{0}$.
Now, let $\mathbf{n_{1}}$ and $\mathbf{n_{2}}$ be two possible values of the vector
$\mathbf{n}$ and, recall that, for $\mathbf{A}$ and $\mathbf{B}$ positive definite matrices,
$\mathbf{A} > \mathbf{B} \Leftrightarrow \mathbf{A} - \mathbf{B} > \mathbf{0}$. Hence, there exists a
function $f$ such that: $f: \mathcal{S}\rightarrow \Re$,
\begin{equation}\label{citerio}
    \widehat{\Cov}(\overline{\mathbf{y}}_{_{ST}}(\mathbf{n_{1}})) <
    \widehat{\Cov}(\overline{\mathbf{y}}_{_{ST}}(\mathbf{n_{2}}))
    \Leftrightarrow
    f\left(\widehat{\Cov}(\overline{\mathbf{y}}_{_{ST}}(\mathbf{n_{1}}))\right)
    < f\left(\widehat{\Cov}(\overline{\mathbf{y}}_{_{ST}}(\mathbf{n_{2}}))\right)
\end{equation}
with $\widehat{\Cov}(\overline{\mathbf{y}}_{_{ST}}(\mathbf{n})) \in \mathcal{S}\subset
\Re^{G(G+1)/2}$ and $\mathcal{S}$ is the set of positive definite matrices.

Then, (\ref{om4}) can be reduced to the following chance constraints mathematical program
\begin{equation}\label{om5}
  \begin{array}{c}
    \build{\min}{}{\mathbf{n}} \mathbf{c}'\mathbf{n}+ c_{0}\\
    \mbox{subject to}\\
    \P\left(f\left(\widehat{\Cov}(\overline{\mathbf{y}}_{_{ST}})\right) \leq \tau\right) \geq p_{0}\\
    2\leq n_{h}\leq N_{h}, \ \ h=1,2,\dots, H\\
    n_{h}\in \mathbb{N}\\
    \Vech \widehat{\Cov}(\overline{\mathbf{y}}_{_{ST}}) \build{\rightarrow}{d}{} \mathcal{N}_{k}
  \left(\E\left(\Vech \widehat{\Cov}(\overline{\mathbf{y}}_{_{ST}})\right),
  \Cov\left(\Vech \widehat{\Cov}(\overline{\mathbf{y}}_{_{ST}})\right)\right).
  \end{array}
\end{equation}
There are many possibilities for the definition of $f(\cdot)$, see \citet{dgu:08}. In particular, it is of interest
when $f = \tr\left(\widehat{\Cov}(\overline{\mathbf{y}}_{_{ST}})\right)$. Among many others
options, it is also interesting the case when $f =
\left|\widehat{\Cov}(\overline{\mathbf{y}}_{_{ST}})\right|$ in (\ref{om5}) which is described
in detail in Section \ref{sec4}, although its application in a real problem poses some
algorithmic and numerical challenges still under study.

\section{Application}\label{sec4}

Lets consider the results of a forest survey conducted in Humboldt
County, California, originally reported in \citet{aa81}. The population was subdivided into nine strata on the basis of the timber
volume per unit area, as determined from aerial photographs. The two variables included in this
example are the basal area (BA)\footnote{In forestry terminology, 'Basal area' is the area of a
plant perpendicular to the longitudinal axis of a tree at 4.5 feet above ground.} in square
feet, and the net volume in cubic feet (Vol.), both expressed on a per acre basis. The
variances, covariances and the number of units within stratum $h$ are listed in Table 1.

\begin{table}
\caption{\small Variances, covariances and the number of units within each stratum}
\begin{center}
\begin{footnotesize}
\begin{minipage}[t]{400pt}
\begin{tabular}{ c r r r r r}
\hline\hline
\multicolumn{2}{c}{} & \multicolumn{3}{c}{Variance} \\
\cline{4-5} Stratum & $N_{h}$ &  $c_{h}$\footnote{These are simulated costs, also $c_0$ is taken
as 0}
&\hspace{.5cm} BA \hspace{.5cm} & \hspace{.5cm} Vol. \hspace{.5cm} & \hspace{.5cm}Covariance \\
\hline\hline
1 & 11 131 & 2.5 & 1 557 & 554 830 & 28 980 \\
2 & 65 857 & 3.0 & 3 575 & 1 430 600 & 61 591\\
3 & 106 936 & 1.5 & 3 163 & 1 997 100 & 72 369 \\
4 & 72 872 & 2.5 & 6 095 & 5 587 900 & 166 120\\
5 & 78 260 & 2.0 & 10 470 & 10 603 000 & 293 960 \\
6 & 51 401 & 2.0 & 8 406 & 15 828 000 & 357 300\\
7 & 24 050 & 2.5 & 20 115 & 26 643 000 & 663 300 \\
8 & 46 113 & 3.0 & 9 718 & 13 603 000 & 346 810\\
9 & 102 985 & 3.5 & 2 478 & 1 061 800 & 39 872 \\
\hline\hline
\end{tabular}
\end{minipage}
\end{footnotesize}
\end{center}
\end{table}

For this example, the matrix optimisation problem under approach (\ref{om5}) is
\begin{equation}\label{ej}
  \begin{array}{c}
  \build{\min}{}{\mathbf{n}}
   \mathbf{n}'\mathbf{c}+c_{0}\\
    \mbox{subject to}\\
    \P\left(f\left(%
    \begin{array}{c c}
      \widehat{\Var}(\overline{y}_{_{ST}}^{1}) & \widehat{\Cov}(\overline{y}_{_{ST}}^{1}, \overline{y}_{_{ST}}^{2})\\
      \widehat{\Cov}(\overline{y}_{_{ST}}^{2}, \overline{y}_{_{ST}}^{1}) & \widehat{\Var}(\overline{y}_{_{ST}}^{2}) \\
    \end{array}%
    \right) \leq \tau\right) \geq p_{0}\\
    \displaystyle\sum_{h=1}^{9}n_{h}=1000 \\
    2\leq n_{h}\leq N_{h}, \ \ h=1,\dots, 9\\
    \widehat{\Cov}(\overline{\mathbf{y}}_{_{ST}}) \build{\rightarrow}{d}{} \mathcal{N}_{2 \times 2}
  \left(\E\left(\widehat{\Cov}(\overline{\mathbf{y}}_{_{ST}})\right),
  \Cov\left(\Vec \widehat{\Cov}(\overline{\mathbf{y}}_{_{ST}})\right)\right)\\
    n_{h}\in \mathbb{N}.
  \end{array}
\end{equation}

\subsection{Solution when $f(\cdot) \equiv \tr(\cdot)$}

Observe that by (\ref{normal}), (\ref{ecyst}) and (\ref{ccyst})
$$
  \tr \Cov\left(\overline{\mathbf{y}}_{ST}\right) \sim \mathcal{N}\left(\E \left (\tr \Cov\left(\overline{\mathbf{y}}_{ST}\right)\right),
  \Var\left(\tr \Cov\left(\overline{\mathbf{y}}_{ST}\right)\right)\right)
$$
where
$$
    \E\left(\tr \widehat{\Cov}(\overline{\mathbf{y}}_{_{ST}})\right) =  \sum_{j=1}^{G}\sum_{h=1}^{H}\left(
     \frac{{{W_{h}}^{2}}}{n_{h}} - \frac{{W_{h}}}{N} \right) \frac{n_{h}}{n_{h}-1}S_{h_{j}}^{2},
$$
$$
    \Var\left(\tr \widehat{\Cov}(\overline{\mathbf{y}}_{_{ST}})\right)
    =\sum_{j=1}^{G}\sum_{h=1}^{H}\left(
    \frac{{{W_{h}}^{2}}}{n_{h}} - \frac{{W_{h}}}{N} \right)^{2} \frac{n_{h}}{(n_{h}-1)^{2}}
    \left(m_{h_{j}}^{4} - (S_{h_{j}}^{2})^{2}\right),
$$
and
$$
  m_{h_{j}}^{4} = \frac{1}{N_{h}}\left[\sum_{i = 1}^{N_{h}} \left(y_{h i}^{j} - \overline{Y}_{h}^{j}\right)^{4} \right].
$$
Standardising the function $f$ in equation (\ref{ej}), it is seen that
$$
    \mbox{P}\left[\frac{\tr \widehat{\Cov}\left(\overline{\mathbf{y}}_{ST}\right)
    -\E(\tr \widehat{\Cov}\left(\overline{\mathbf{y}}_{ST}\right))}
    {\sqrt{\Var(\tr \widehat{\Cov}\left(\overline{\mathbf{y}}_{ST}\right))}} \leq
    \frac{\tau-\E(\tr \widehat{\Cov}\left(\overline{\mathbf{y}}_{ST}\right))}
    {\sqrt{\Var(\tr \widehat{\Cov}\left(\overline{\mathbf{y}}_{ST}\right))}}\right]
    \geq p_{0},
$$
with
$$
  p_{0}= \Phi\left(\frac{\tau-\E(\tr \widehat{\Cov}\left(\overline{\mathbf{y}}_{ST}\right))}
  {\sqrt{\Var(\tr \widehat{\Cov}\left(\overline{\mathbf{y}}_{ST}\right))}}\right),
$$
where $\Phi(\cdot)$, denotes the standard normal distribution function. Let $e_{p_{0}}$
be the value of the standard normal random variable such that $\Phi(e_{p_{0}})=p_{0}$, in such
way that the inequality can be established as
$$
  \Phi\left(\frac{\tau-\E(\tr \widehat{\Cov}\left(\overline{\mathbf{y}}_{ST}\right))}
  {\sqrt{\Var(\tr \widehat{\Cov}\left(\overline{\mathbf{y}}_{ST}\right))}}\right)\geq \Phi(e_{p_{0}}),
$$
which holds only if
$$
     \frac{\tau-\E(\tr \widehat{\Cov}\left(\overline{\mathbf{y}}_{ST}\right))}
     {\sqrt{\Var(\tr \widehat{\Cov}\left(\overline{\mathbf{y}}_{ST}\right))}}\geq e_{p_{0}},
$$
or equivalently
\begin{equation}\label{102}
\E(\tr \widehat{\Cov}\left(\overline{\mathbf{y}}_{ST}\right))+ e_{p_{0}} \sqrt{\Var(\tr
\widehat{\Cov}\left(\overline{\mathbf{y}}_{ST}\right))}-\tau\leq 0.
\end{equation}
Hence, taking into account (\ref{sus}), the equivalent deterministic problem to the stochastic
mathematical programming (\ref{ej}), is given by
$$
  \begin{array}{c}
  \build{\min}{}{\mathbf{n}}
    \mathbf{n}'\mathbf{c}+c_{0}\\
    \mbox{subject to}\\
    \widehat{\E}(\tr \widehat{\Cov}\left(\overline{\mathbf{y}}_{ST}\right))+ e_{p_{0}} \sqrt{\widehat{\Var}(\tr
    \widehat{\Cov}\left(\overline{\mathbf{y}}_{ST}\right))}-\tau\leq 0\\
    \displaystyle\sum_{h=1}^{9}n_{h}=1000 \\
    2\leq n_{h}\leq N_{h}, \ \ h=1,\dots, 9\\
    n_{h}\in \mathbb{N}.
  \end{array}
$$
where
\begin{equation}\label{esp}
    \widehat{\E}\left(\tr \widehat{\Cov}(\overline{\mathbf{y}}_{_{ST}})\right) =  \sum_{j=1}^{G}\sum_{h=1}^{H}\left(
     \frac{{{W_{h}}^{2}}}{n_{h}} - \frac{{W_{h}}}{N} \right) \frac{n_{h}}{n_{h}-1}s_{h_{j}}^{2},
\end{equation}
\begin{equation}\label{varia}
    \widehat{\Var}\left(\tr \widehat{\Cov}(\overline{\mathbf{y}}_{_{ST}})\right)
    =\sum_{j=1}^{G}\sum_{h=1}^{H}\left(
    \frac{{{W_{h}}^{2}}}{n_{h}} - \frac{{W_{h}}}{N} \right)^{2} \frac{n_{h}}{(n_{h}-1)^{2}}
    \left(m_{h_{j}}^{4} - (s_{h_{j}}^{2})^{2}\right),
\end{equation}
and
\begin{equation}\label{impo}
    m_{h_{j}}^{4} = \frac{1}{n_{h}}\left[\sum_{i = 1}^{n_{h}} \left(y_{h i}^{j} - \overline{y}_{h}^{j}\right)^{4} \right].
\end{equation}
\begin{rem}\label{impo1}
Observe that the estimators $\overline{y}_{h}^{j}$, $s_{h_{j}}^{2}$ and $m_{h_{j}}^{4}$ of
$\overline{Y}_{h}^{j}$, $S_{h_{j}}^{2}$ and $M_{h_{j}}^{4}$  could initially be obtained as
\begin{description}
  \item[i)] results from a pilot (preliminary) sample or
  \item[ii)] using the corresponding values of the estimators from another set of variables, $X$'s,
  correlated to the variables $Y$'s.
\end{description}
It is important to have this in mind in the  minimisation step, because for example,
the $n_{h}$ that appears in expression (\ref{impo}), is the value of $n_{h}$ (fixed) used in the
pilot study. Same comment for the expression of the estimator $\overline{y}_{h}^{j}$ and
$s_{h_{j}}^{2}$. While the $n_{h}$'s that appear in expressions (\ref{esp}) and (\ref{varia})
are the decision variables. \qed
\end{rem}

% \begin{rem}
\subsection{Solution when $f(\cdot) \equiv |\cdot|$}

Assume the following alternative stochastic matrix mathematical programming problem
\begin{equation}\label{det1}
  \begin{array}{c}
  \build{\min}{}{\mathbf{n}}
    \mathbf{n}'\mathbf{c}+c_{0}\\
    \mbox{subject to}\\
    \P\left(f\left(\widehat{\Cov}(\overline{\mathbf{y}}_{_{ST}})
    \right) \leq \tau\right) \geq p_{0}\\
    \displaystyle\sum_{h=1}^{9}n_{h}=1000 \\
    2\leq n_{h}\leq N_{h}, \ \ h=1,\dots, 9\\
    \Vech \widehat{\Cov}(\overline{\mathbf{y}}_{_{ST}}) \build{\rightarrow}{d}{} \mathcal{N}_{G \times G}
    \left(\Vech \mathbf{0}_{G \times G},
    \Cov\left(\Vech\widehat{\Cov}(\overline{\mathbf{y}}_{_{ST}})\right)\right)\\
    n_{h}\in \mathbb{N},
\end{array}
\end{equation}
where $\widehat{\Cov}(\overline{\mathbf{y}}_{_{ST}})$
\begin{small}
$$
   = \Vech^{-1}\left[\Vech\widehat{\Cov}(\overline{\mathbf{y}}_{_{ST}})
  - \E\left(\Vech\widehat{\Cov}(\overline{\mathbf{y}}_{_{ST}})\right)\right]
$$
\end{small}
and $\Vech^{-1}$ is the inverse function of function $\Vech$.

Then, the restriction in (\ref{det1}), is
$$
  \P\left(\left|\widehat{\Cov}(\overline{\mathbf{y}}_{_{ST}})
    \right| \leq \tau\right) \geq p_{0}
$$
which for $G = 2$ and assuming that
$\widehat{\Cov}\left(\Vech\widehat{\Cov}(\overline{\mathbf{y}}_{_{ST}})\right)$ is
such that
$$
  \widehat{\Cov}\left(\Vech\widehat{\Cov}(\overline{\mathbf{y}}_{_{ST}})\right) = \mathbf{B}\otimes
  \mathbf{B} = \mathbf{N},
$$
implies that
$$
  \P\left(\left|\widehat{\Cov}(\overline{\mathbf{y}}_{_{ST}})
    \right| \leq \tau |\mathbf{N}|^{1/4}\right) \geq p_{0}
$$
where
$$
  \mathbf{N} = \sum_{h=1}^{H}\left(
    \frac{{{W_{h}}^{2}}}{n_{h}} - \frac{{W_{h}}}{N} \right)^{2} \frac{n_{h}}{(n_{h}-1)^{2}}
    \left(\mathbf{\mathfrak{m}}_{h}^{4} - \Vec \mathbf{s}_{h}\Vec' \mathbf{s}_{h}\right),
$$
$$
  \mathbf{\mathfrak{m}}_{h}^{4} = \frac{1}{n_{h}}\left[\sum_{i = 1}^{n_{h}}
    (\mathbf{y}_{h i} - \overline{\mathbf{y}}_{h})(\mathbf{y}_{h i} - \overline{\mathbf{y}}_{h})'
    \otimes (\mathbf{y}_{h i} - \overline{\mathbf{y}}_{h})(\mathbf{y}_{h i} - \overline{\mathbf{y}}_{h})'
    \right].
$$
see Remark \ref{impo1}, and
$$
  p_{0}= \Psi\left(\tau |\mathbf{N}|^{1/4}\right),
$$
with $\Psi(\cdot)$, denotes the distribution function of
$\left|\widehat{\Cov}(\overline{\mathbf{y}}_{_{ST}})\right|$, see \citet{dc:00}. Let
$r_{p_{0}}$ be the percentile of a random variable such that $\Psi(r_{p_{0}})=p_{0}$, in such
way that the inequality can be established as
$$
  \Psi\left(\left|\widehat{\Cov}(\overline{\mathbf{y}}_{_{ST}})
    \right| \leq \tau |\mathbf{N}|^{1/4}\right)\geq \Psi(r_{p_{0}}),
$$
which holds only if
$$
  \tau|\mathbf{N}|^{1/4}\geq r_{p_{0}},
$$
where the density of $Z = \widehat{\Cov}(\overline{\mathbf{y}}_{_{ST}})$ is, see
\citet{dc:00}
$$
  \frac{dG(z)}{dz} = g_{_{Z}}(z) = \frac{1}{\sqrt{2}} \exp (z)\left[1 - \erf\left(\sqrt{2z}\right)\right], \quad
  z\geq 0,
$$
where $\erf(\cdot)$ is the usual error function defined as
$$
  \erf(x) = \frac{2}{\sqrt{\pi}}\int_{0}^{x} \exp(-t^{2})dt.
$$
Thus, by (\ref{sus}), the equivalent deterministic problem to the stochastic mathematical
programming problem (\ref{det1}), is given by
$$
 \begin{array}{c}
  \build{\min}{}{\mathbf{n}}
    \mathbf{n}'\mathbf{c}+c_{0}\\
    \mbox{subject to}\\
    \tau|\mathbf{N}|^{1/4}\geq r_{p_{0}}\\
    \displaystyle\sum_{h=1}^{9}n_{h}=1000 \\
    2\leq n_{h}\leq N_{h}, \ \ h=1,\dots, 9\\
    n_{h}\in \mathbb{N},
\end{array}
$$
%\end{rem}

Table 2 includes the optimum allocation for each characteristic, BA and Vol (the second and
third rows) from a deterministic point of view. Also appear (on fourth and fifth rows) the
optimal allocations via the deterministic problem (\ref{om2}), identified in the table with the
name Pr\'ekopa, and the deterministic version of (\ref{om5}) when $f(\cdot) = \tr(\cdot)$.
These results are presented in their stochastic version in the 7-10th rows. The last three
columns show the minimum values of the individual variances for the respective optimum
allocations and the cost identified by each method. The results were computed using the
commercial software Hyper LINGO/PC, release 6.0, see \citet{w95}. The default optimisation
methods used by LINGO to solve the nonlinear integer optimisation programs are Generalised
Reduced Gradient (GRG) and branch-and-bound methods, see \citet{bss06}. Finally, note that,
for this sampling study, there is not a great discrepancy between the different methods among the sizes of the
strata. And for the multivariate solutions, the biggest cost difference appears in the
deterministic version of Pr\'ekopa's method.

\begin{table}
\caption{\small Sample sizes and estimator of variances for the different allocation
rules}
\begin{center}
\begin{minipage}[t]{400pt}
\begin{scriptsize}
\begin{tabular}{ c  c  c  c  c  c  c c c c c cc}
\hline\hline Allocation & $n_{1}$ & $n_{2}$ & $n_{3}$ & $n_{4}$ & $n_{5}$ & $n_{6}$ & $n_{7}$ &
$n_{8}$ & $n_{9}$ & $\widehat{\Var}(\overline{y}_{_{ST}}^{1})$ &
$\widehat{\Var}(\overline{y}_{_{ST}}^{2})$ & Cost\\
\hline\hline %
\multicolumn{12}{c}{\textbf{Deterministic approach}}\\
BA\footnote{\scriptsize With $v_{0}^{1} = 6$}
& 10 & 78 & 171 & 123 & 194 & 114 & 75 & 90 & 94 & 5.599 & 5766.161 & 2225.5\\
Vol\footnote{\scriptsize With $v_{0}^{2} = 6000$}
& 6 & 51 & 139 & 123 & 204 & 163 & 90 & 109 & 64 & 6.502 & 5499.996 & 2194.0\\
Pr\'ekopa\footnote{\scriptsize With $v_{0}^{1} = 6$ and With $v_{0}^{2} = 6000$}
& 10 & 78 & 171 & 123 & 194 & 114 & 75 & 90 & 94 & 5.599 & 5766.161 & 2225.5\\
$\tr \widehat{\Cov}(\overline{\mathbf{y}}_{_{ST}})$\footnote{\scriptsize With $\tau = 6000$}
& 6 & 47 & 127 & 114 & 186 & 149 & 80 & 102 & 59 & 7.071 & 5992.921 & 2014.0\\
\multicolumn{12}{c}{\textbf{Stochastic approach}\footnote{\scriptsize With $p_{0} = 0.50$}}\\
BA & 10 & 79 & 168 & 125 & 196 & 117 & 76 & 91 & 95 & 5.939 & 5693.354 & 2248.0\\
Vol & 6 & 48 & 129 & 113 & 189 & 150 & 82 & 102 & 60 & 6.988 & 5933.759 & 2034.0\\
Pr\'ekopa & 11 & 79 & 169 & 123 & 196 & 117 & 78 & 91 & 96 & 5.921 & 5680.571 & 2034.0\\
$\tr \widehat{\Cov}(\overline{\mathbf{y}}_{_{ST}})$
& 6 & 48 & 129 & 114 & 188 & 151 & 81 & 102 & 60 & 6.988 & 5933.752 & 2034.0\\
\hline\hline
\end{tabular}
\end{scriptsize}
\end{minipage}
\end{center}
\end{table}

\section{\normalsize Conclusions}

There is a vast literature on the problem of sample allocation in stratified sampling.
A natural approach considers a cost minimisation problem subject to variance restrictions.
This paper follows Pr\'ekopa's approach by setting the problem into the area of stochastic optimization.
It is recognized that this is a more realistic approach because, in general, the
population variances of the strata are unknown and therefore requires estimating
them. As a result, problem (\ref{om2}) really falls within the scope of stochastic
mathematical programming which incorporates the inherent uncertainty of estimators in a natural way.

The approach is not without its drawbacks, it is difficult to give general rules for electing
the value function $ f (\cdot) $, potentially there are an infinite number of possibilities.
In this paper we have chosen to work with $f(\mathbf{A}) = |\mathbf{A}|$ and $f(\mathbf{A}) = \text{tr}(\mathbf{A})$
which can be interpreted as a generalised variance and as an average variance respectively.
However, the responsibility for the selection
or definition of that function, lies wholly with the expert in the field of application.

\section*{Acknowledgments}

%The authors wish to thank the Editor and the anonymous reviewers for their constructive
%comments on the preliminary version of this paper.
This research work was partially supported by IDI-Spain, Grants No. FQM2006-2271 and
MTM2008-05785, supported also by CONACYT Grant CB2008 Ref. 105657. This paper was written
during J. A. D\'{\i}az-Garc\'{\i}a's stay as a visiting
professor at the Department of Probability Statistics of the Center of Mathematical Research,
Guanajuato, M\'exico.

\end{document}